\documentclass[11pt,a4paper,twoside]{article}
\usepackage[hscale=0.8,vscale=0.7]{geometry}
\usepackage[utf8]{inputenc}
\usepackage[english]{babel}
\usepackage{amsmath,amssymb,amsthm}
\usepackage{amsfonts}
\usepackage{dsfont}
\usepackage{marvosym}
\usepackage{graphicx} 

\numberwithin{equation}{section}

\newcommand{\N}{\mathbb{N}}
\newcommand{\Z}{\mathbb{Z}}
\newcommand{\Q}{\mathbb{Q}}

\newcommand{\R}{\mathbb{R}}
\newcommand{\PR}{\mathbb{P}}

\newcommand{\Za}[1]{\mathrm{Z}_{\alpha} \left( {#1} \right) }

\newtheorem{Theo}{Theorem}[section]
\newtheorem{Lem}{Lemma}[section]
\newtheorem{Prop}{Proposition}[section]

\newtheorem{Rem}{Remark}[section]

\def\al{\alpha}

\newcommand{\supp}{\mathrm{supp}}

\def\al{\alpha}

\title{Linear fractional stable motion: a wavelet estimator of the $\al$ parameter}

\author{Antoine Ayache \\UMR CNRS 8524, Laboratoire Paul Painlev\'e, B\^at. M2\\
  Universit\'e Lille 1\\ 59655 Villeneuve d'Ascq Cedex, France\\
E-mail: \texttt{Antoine.Ayache@math.univ-lille1.fr}\\
\ 
\and
 Julien Hamonier\\
  FR CNRS 2956,  LAMAV, \\
  Institut des Sciences et Techniques de Valenciennes,\\
  Universit\'e de Valenciennes et du Hainaut Cambr\'esis, \\
  F-59313 - Valenciennes Cedex 9, France\\
E-mail: \texttt{Julien.Hamonier@univ-valenciennes.fr}
}

\date{}

\begin{document}
\maketitle
\begin{abstract}
Linear fractional stable motion, denoted by $\{X_{H,\al}(t)\}_{t\in \R}$, is one of the most classical stable processes; it depends on two 
parameters $H\in (0,1)$ and $\al\in (0,2)$. The parameter $H$ characterizes the self-similarity property of $\{X_{H,\al}(t)\}_{t\in \R}$
while the parameter $\al$ governs the tail heaviness of its finite dimensional distributions; throughout our article we assume that the latter distributions 
are symmetric, that
$H>1/\al$ and that $H$ is known. We show that, on the interval $[0,1]$, the asymptotic behaviour of the maximum, at a given scale $j$, of absolute values of the wavelet coefficients of $\{X_{H,\al}(t)\}_{t\in \R}$, is of the same order as $2^{-j(H-1/\al)}$; then we derive from this result a strongly consistent (i.e. almost surely 
convergent) statistical estimator for the parameter~$\al$.
\end{abstract}

\medskip


\section{Introduction and statement of the main results}

Let $H$ and $\al$ be two parameters such that $\al\in (1,2)$ and $1/\al <H <1$. We denote by $\{X_{H,\al}(t)\}_{t\in \R}$ the symmetric $\al$ stable linear fractional stable motion (lfsm for brevity) (see e.g. \cite{SamTaq,EmMa}), defined, for
all $t\in\R$, as, 
\begin{equation}
X_{H,\al}(t):=\int_{\R} \Big\{ (t-s)_+^{H-1/\al} - (-s)_+^{H-1/\al} \Big\} \Za{ds},
\end{equation}
where $\Za{\cdot}$ is a symmetric $\al$-stable random measure and, for each $z\in\R$
\begin{equation}
\label{eq:pospart}
(z)_+:=\max\{z,0\}. 
\end{equation}
The parameter $H$ characterizes the
self-similarity property of lfsm; namely, for all fixed positive real-number
$a$, the processes $\{X(at)\}_{t\in\R}$ and $\{a^H X(t)\}_{t\in\R}$ have the
same  finite dimensional distributions. The parameter $\al$ governs the tail heaviness
of the latter distributions. The
process $\{X_{H,\al}(t)\}_{t\in \R}$ has a modification with continuous
nowhere differentiable sample
paths; it is identified with this modification in all the sequel. 

The statistical problem of the
estimation of $H$ has already been studied in several articles: 
\cite{SPT2002, abry1999estimation, stoev2005asymptotic, pipiras2007bounds}, and strongly consistent estimators
(i.e. convergent almost surely), based
on $(d_{j,k})_{(j,k)\in\Z^2}$, the discrete wavelet transform of lfsm, have
been proposed; notice that the latter estimators of $H$ do not require that 
$\al$ to be known. Throughout our paper, for all $(j,k)\in\Z^2$, the wavelet coefficient $d_{j,k}$ is
defined as,
\begin{equation}\label{djk}
d_{j,k}=2^j \int_{\R} X_{H,\al}(t) \psi(2^jt-k) dt;
\end{equation}
moreover, we only impose to the analyzing wavelet $\psi$ a
very weak assumption: {\em $\psi$ is an arbitrary real-valued non-vanishing continuous function with 
a compact support in $[0,1]$ and it has 2 vanishing moments i.e.
\begin{equation}\label{moment}
\int_{\R} \psi(s) ds=\int_{\R} s\psi(s) ds=0.
\end{equation}
}
It is worth noticing that we do not need that
$\big\{2^{j/2}\psi(2^j\cdot-k):(j,k)\in\Z^2\big\}$ be an orthonormal wavelet basis for
$L^2(\R)$.

In view of the fact that the problem of the estimation of $H$ is now well
understood, from now on we assume the latter parameter to be known. Our goal is
to construct, by using the wavelet coefficients $(d_{j,k})_{0\le k <2^j}$,
a strongly consistent (i.e. almost surely convergent when $j\rightarrow +\infty$)
estimator $\widehat{\al}_j$ of the parameter $\al$. Let us outline the main
ideas which lead to this estimator. 
\begin{itemize}
\item The starting point, is a result of \cite{Tak89}, according to which, with probability $1$, the quantity
$H-1/\al$, is the critical uniform H\"older exponent of the sample
paths of $X_{H,\al}$ over any arbitrary compact interval and in particular the
interval $[0,1]$; more precisely, one has, almost surely for all arbitrarily
small $\eta>0$,
\begin{equation}
\label{eq:modcont}
\sup_{t_1,t_2\in [0,1]}
\left\{\frac{\big|X_{H,\al}(t_1)-X_{H,\al}(t_2)\big|}{|t_1-t_2|^{H-1/\al-\eta}}\right\}<\infty
\end{equation}
and
\begin{equation}
\label{eq:optmodcont}
\sup_{t_1,t_2\in [0,1]} \left\{\frac{\big|X_{H,\al}(t_1)-X_{H,\al}(t_2)\big|}
{|t_1-t_2|^{H-1/\al+\eta}}\right\}=\infty.
\end{equation}
\item Next, let us set,
\begin{equation}
\label{eq:wavc-D}
D_j=\max_{0\le k <2^j}|d_{j,k}|.
\end{equation}
In view of the fact that the wavelet $\psi$ has a first vanishing moment, one
can derive from (\ref{eq:modcont}), that, almost surely, for all arbitrarily
small $\epsilon>0$,
\begin{equation}
\label{eq:mainub}
\limsup_{j\rightarrow +\infty}\left\{ 2^{j(H-1/\al-\epsilon)}D_j \right\}<\infty.
\end{equation}
\item Notice that, since we do not impose to $\psi$ to be a continuously differentiable function and to
  $\big\{2^{j/2}\psi(2^j\cdot-k):(j,k)\in\Z^2\big\}$ to form an orthonormal wavelet basis for
$L^2(\R)$, a priori it is not at all clear that (\ref{eq:optmodcont}), implies that, almost
surely for all arbitrarily
small $\epsilon>0$,
\begin{equation}
\label{eq:lb1}
\limsup_{j\rightarrow +\infty}\left\{ 2^{j(H-1/\al+\epsilon)}D_j\right\}=\infty.
\end{equation}
Yet, by making use of some specific properties of lfsm as well as the fact
that $\psi$ is compactly supported, we will be able to show that, a result
stronger than (\ref{eq:lb1}) holds; namely, one has almost surely, for all
arbitrarily small
$\epsilon>0$,
\begin{equation}
\label{eq:mainlb}
\liminf_{j\rightarrow +\infty}\left\{2^{j(H-1/\al+\epsilon)}D_j\right\}=\infty.
\end{equation}
\item Finally, combining (\ref{eq:mainub}) with (\ref{eq:mainlb}), one can
  get the following theorem, which is our main result.
\end{itemize}
\begin{Theo}
\label{Th:main}
For each  $j\in \N$, one set,
$$
\frac{1}{\widehat{\al}_j}=H+\frac{\log (D_j)}{j\log (2)},
$$
where $D_j$ is defined in (\ref{eq:wavc-D}). Then, one has almost surely,
$$
\widehat{\al}_j \xrightarrow[j\rightarrow+\infty]{a.s.} \al.
$$
\end{Theo}

\section{Proofs}

\subsection{Proof of Relation (\ref{eq:mainub})}

The proof is standard in the wavelet setting, we give it for the sake of completeness. Let $\check{\Omega}$ be an event of probability 1 on which Relation
(\ref{eq:modcont}) holds and let $\omega\in\check{\Omega}$ be arbitrary and
fixed. Assume that $\epsilon>0$ is arbitrary and fixed and denote by $C(\omega)$ the finite quantity defined as, 
\begin{equation}
\label{eq:defC}
C(\omega):=\sup_{t_1,t_2\in [0,1]}\left\{\frac{\big|X_{H,\al}(t_1,\omega)-X_{H,\al}(t_2,\omega)\big|}{|t_1-t_2|^{H-1/\al-\epsilon}}\right\}.
\end{equation}

On the other hand, notice that (\ref{djk}), (\ref{moment}) and the 
fact that,
\begin{equation}
\label{eq:supp}
\supp \, \psi \subseteq [0,1],
\end{equation}
 imply that, for all
$(j,k)\in\Z_+\times\Z_+$ satisfying $0\le k < 2^j$, one has,
\begin{equation}
\label{eq:ub3}
d_{j,k}(\omega)= 2^j \int_{k2^{-j}}^{(k+1)2^{-j}} \Big\{ X_{H,\al}(t,\omega)-X_{H,\al}(k2^{-j},\omega) \Big\}\psi(2^jt-k) dt.
\end{equation}
Next, combining (\ref{eq:ub3}) with (\ref{eq:defC}), one gets
\begin{eqnarray*}
 |d_{j,k}(\omega)|&\leq &  2^j\int_{k2^{-j}}^{(k+1)2^{-j}} \Big| X_{H,\al}(t,\omega)-X_{H,\al}(k2^{-j},\omega) \Big| \big|\psi(2^jt-k)\big| dt\nonumber\\
&\le & \|\psi\|_{L^\infty (\R)}C(\omega) 2^j\int_{k2^{-j}}^{(k+1)2^{-j}} \big
|t-k2^{-j}\big|^{H-1/\al-\epsilon}dt\nonumber\\
&\le & \|\psi\|_{L^\infty (\R)}C(\omega)  2^{-j(H-1/\al-\epsilon)},
\end{eqnarray*}
which proves that (\ref{eq:mainub}) is satisfied. $\Box$

\subsection{Proof of Relation (\ref{eq:mainlb})}

Let us first recall that in \cite{delbeke2000stochastic}, a nice stochastic
integral representation of the wavelet coefficients $d_{j,k}$ has been
obtained, namely one has almost surely that
\begin{equation}
\label{eq:lb}
d_{j,k}  = 2^{-j(H-1/\al)} \int_{\R} \Phi_{H,\al}(2^js-k) \Za{ds},
\end{equation}
where $\Phi_{H,\al}$ is the real-valued continuous function defined for each $x\in\R$, as,
\begin{equation}\label{PhiH}
\Phi_{H,\al}(x)=\int_{\R} (y-x)_+^{H-1/\al} \psi(y) dy=\int_{0}^1 (y-x)_+^{H-1/\al} \psi(y) dy;
\end{equation}
notice that the last equality results from (\ref{eq:supp}).

\begin{Prop}\label{localisation}
The function $\Phi_{H,\al}$ satisfies the following two nice properties:
\begin{itemize}
\item[(i)] one has,
\begin{equation}
\label{eq:suppPhiH}
\supp\, \Phi_{H,\al} \subseteq (-\infty,1];
\end{equation}
\item[(ii)] there is a constant $c_1>0$ such for all $x\in (-\infty,1]$,
\begin{equation}
\label{ineg:loc}
\big|\Phi_{H,\al}(x)\big|\le c_1 \big(1+|x|\big)^{-(2+1/\al-H)}.
\end{equation}
\end{itemize}
\end{Prop}

\begin{proof}[Proof of Proposition~\ref{localisation}] Part~$(i)$ is a straightforward consequence of (\ref{PhiH}) and (\ref{eq:pospart}). Let us 
show that Part~$(ii)$ holds. First observe that, (\ref{PhiH}) easily implies that,
\begin{equation}
\label{eq0:localisation}
\sup_{x\in [-1,1]}\left\{\big(1+|x|\big)^{2+1/\al-H}\big|\Phi_{H,\al}(x)\big|\right\}\le 4\|\psi\|_{L^\infty (\R)}<\infty.
\end{equation}
Let us now suppose that $x<-1$. We denote by $\psi^{(-1)}$ the primitive of $\psi$,
defined for all $z\in\R$, as
$$
\psi^{(-1)}(z)=\int_{-\infty}^z \psi(y) dy.
$$ 
Observe that (\ref{eq:supp}) and (\ref{moment}) entail that the continuous function
$\psi^{(-1)}$ has a compact support included in $[0,1]$. We denote by $\psi^{(-2)}$ the primitive of $\psi^{(-1)}$,
defined for all $z\in\R$, as
$$
\psi^{(-2)}(z)=\int_{-\infty}^z \psi^{(-1)}(y) dy.
$$ 
Observe that $\mbox{supp}\, \psi^{(-1)}\subseteq [0,1]$ and (\ref{moment}) entail that the continuous function
$\psi^{(-2)}$ has a compact support included in $[0,1]$; therefore integrating two times by parts in (\ref{PhiH}), we obtain
\begin{equation}
\label{eq2:localisation}
\Phi_{H,\al}(x)=(H-1/\al)(H-1/\al-1)\int_{0}^1 (y-x)^{H-1/\al-2}\psi^{(-2)}(y)dy.
\end{equation}
Next, using (\ref{eq2:localisation}) and the inequalities: for all
$y\in [0,1]$, $y-x\ge |x|\ge 2^{-1}\big (1+|x|\big)$,  it follows that,
\begin{equation}
\label{eq1:localisation}
\big|\Phi_{H,\al}(x)\big|\le 2^{2+1/\al-H}\|\psi^{(-2)}\|_{L^\infty(\R)} \big(1+|x|\big)^{H-1/\al-2}.
\end{equation}
Finally, combining (\ref{eq0:localisation}) with (\ref{eq1:localisation}), we get Part~$(ii)$ of the proposition.
\end{proof}
\noindent
A straightforward consequence of (\ref{eq:lb}) and Part~$(i)$ of Proposition~\ref{localisation}, is that,
\begin{equation}
\label{eq:decomp-djk}
d_{j,k}=2^{-j(H-1/\al)} \int_{-\infty}^{(k+1)2^{-j}} \Phi_{H,\al}(2^js-k) \Za(ds).
\end{equation}

Let us now introduce some additional notations. 
We assume that $\delta \in (0,1/3)$ is arbitrary and fixed. For all $j\in\Z_+$,
we define the positive integer $e_j$ as,
\begin{equation}
\label{eq:add1}
e_j:=[2^{j\delta}],
\end{equation}
where $[\cdot]$ is the integer part function. Then, for any integer $l$ such
that 
\begin{equation}
\label{eq:inegl}
0\leq l \le [2^{j(1-\delta)}]-1,
\end{equation} 
we set
\begin{equation}\label{def:Gjl}
G_{j,le_j}:= \int_{((l-1)e_j+1)2^{-j}}^{(le_j+1)2^{-j}} \Phi_{H,\al}(2^js-le_j) \Za(ds),
\end{equation}
and
\begin{equation}\label{def:Rjl}
R_{j,le_j}:=\int_{-\infty}^{((l-1)e_j+1)2^{-j}} \Phi_{H,\al}(2^js-le_j) \Za(ds).
\end{equation}
Thus, in view of (\ref{eq:decomp-djk}), the wavelet coefficient $d_{j,le_j}$ can be expressed as,
\begin{equation}\label{rel:d:GR}
d_{j,le_j}=2^{-j(H-1/\al)} \Big(G_{j,le_j} + R_{j,le_j} \Big).
\end{equation}

Now, our goal will be to derive the following two lemmas which respectively provide lower and upper asymptotic estimates for $\max_{0\leq l < [2^{j(1-\delta)}]} |G_{j,le_j}|$ and $\max_{0\leq l < [2^{j(1-\delta)}]}|R_{j,le_j}|$.

\begin{Lem}\label{prop:minmaxgjl}
One has, almost surely
\begin{equation}
\label{eq0:minmaxgjl}
\liminf_{j\rightarrow +\infty} \left\{2^{j\frac{2\delta}{\al}} \max_{0\leq l < [2^{j(1-\delta)}]} |G_{j,le_j}|\right\} \ge 1.
\end{equation}
\end{Lem}

\begin{Lem}\label{prop:majmaxrjl}
One has, almost surely
\begin{equation}
\label{eq0:majmaxrjl}
\limsup_{j\rightarrow +\infty} \left\{ 2^{j\frac{2\delta}{\al}}\max_{0\leq l < [2^{j(1-\delta)}]} |R_{j,le_j}| \right\}= 0.
\end{equation}
\end{Lem}

The proof of Lemma~\ref{prop:minmaxgjl} mainly relies on the following two results.
\begin{Lem}
\label{lem:tailZ}
(see e.g. \cite{SamTaq}) Let $Y$ be an arbitrary symmetric $\al$-stable random variable with a non-vanishing scale parameter $\|Y\|_{\al}$, then for any real number $t\geq \|Y\|_{\al}$, one has,
\begin{equation}\label{eq2:minmaxgjl}
c_3 \|Y\|_{\al}^{\al} t^{-\al} \leq \PR(|Y|>t) \leq c_2 \|Y\|_{\al}^{\al} t^{-\al},
\end{equation}
where $c_2$ and $c_3$ are two positive constants only depending on $\al$.
\end{Lem}

\begin{Lem}\label{lem:gjlrjl}
For each fixed $j\in\Z_+$, $\{G_{j,le_j}:0\leq l \leq [2^{j(1-\delta)}]-1 \}$ is a sequence of identically distributed independent symmetric $\al$-stable random variables whose scale parameters, denoted $\|G_{j,le_j}\|_{\al}$, satisfy for all $l$, 
\begin{equation}
\label{eq1:lem:gjlrjl}
\|G_{j,le_j}\|_{\al}^{\al} =  2^{-j}\int_{1-e_j}^{1} |\Phi_{H,\al}(x)|^{\al} dx.
\end{equation}
\end{Lem}

\begin{proof}[Proof of Lemma \ref{lem:gjlrjl}]
The independence of these symmetric $\al$-stable random variables is a straightforward consequence of the fact that they are defined (see (\ref{def:Gjl})) through stable stochastic integrals 
over disjoint intervals. In order to show that they are identically distributed it is sufficient to prove that (\ref{eq1:lem:gjlrjl}) holds for each $l$. Using a standard property of stable stochastic integrals (see e.g. \cite{SamTaq}) and (\ref{def:Gjl}), one gets
$$
\|G_{j,le_j}\|_{\al}^{\al} = \int_{((l-1)e_j+1)2^{-j}}^{(le_j+1)2^{-j}} |\Phi_{H,\al}(2^js-le_j)|^{\al} ds;
$$
then the change of variable $u=2^js-le_j$ allows to obtain (\ref{eq1:lem:gjlrjl}).
\end{proof}

Now, we are in position to prove Lemma~\ref{prop:minmaxgjl}.

\begin{proof}[Proof of Lemma~\ref{prop:minmaxgjl}]
Let $j\in\Z_+$ be arbitrary and fixed. Using the fact that $\{G_{j,le_j}:0\leq l \leq [2^{j(1-\delta)}]-1 \}$ is a sequence of independent identically distributed random variables (see Lemma~\ref{lem:gjlrjl}), one gets, 
\begin{align}\label{eq1:minmaxgjl}
& \PR \Big( \max_{0\leq l < [2^{j(1-\delta)}]} |G_{j,le_j}| \leq 2^{-j\frac{2\delta}{\al}}
 \Big)= \prod_{l=0}^{[2^{j(1-\delta)}]-1}\PR \Big( |G_{j,le_j}| \leq 2^{-j\frac{2\delta}{\al}}\Big)\nonumber \\
& = \PR\Big(|G_{j,0}| \leq 2^{-j\frac{2\delta}{\al}} \Big)^{[2^{j(1-\delta)}]} = \bigg( 1 -\PR\Big(|G_{j,0}| > 2^{-j\frac{2\delta}{\al}}\Big)\bigg)^{[2^{j(1-\delta)}]}.
\end{align}
Observe that, in view of (\ref{eq1:lem:gjlrjl}) in which one takes $l=0$ and in view of the assumption that $\delta\in (0,1/3)$, there exist a positive constant $c_4$ and a positive integer $j_0$ such that, one has,
\begin{equation}\label{eq2bis:minmaxgjl}
2^{-j\frac{2\delta}{\al}}\ge \|G_{j,le_j}\|_{\al}=2^{-j/\al}\left(\int_{1-e_j}^{1} |\Phi_{H,\al}(x)|^{\al} dx\right)^{1/\al}\ge c_{4}^{1/\al} 2^{-j/\al},
\end{equation}
 for all integers $j$ and $l$ satisfying $j\ge j_0$ and $0\le l < [2^{j(1-\delta)}]$; notice that the last inequality in (\ref{eq2bis:minmaxgjl}), follows from the 
 fact that we have chosen $j_0$, such that for every $j\ge j_0$,
 $$
 \int_{1-e_j}^{1} |\Phi_{H,\al}(x)|^{\al} dx\ge 2^{-1}\int_{-\infty}^{1} |\Phi_{H,\al}(x)|^{\al} dx,
 $$
 and the last integral is positive since $\Phi_{H,\al}$ is a non-vanishing function (this is a consequence of our assumptions on $\psi$). Also notice that, one can suppose that $c_4\in \big(0,c_{3}^{-1}\big)$ (the positive constant $c_3$ has been introduced in Lemma~\ref{lem:tailZ}). Next, it follows from (\ref{eq1:minmaxgjl}), from the first inequality in (\ref{eq2:minmaxgjl}) in which one $t=2^{-j\frac{2\delta}{\al}}$, and from (\ref{eq2bis:minmaxgjl}), that, for all integer $j\ge j_0$,
\begin{equation}\label{eq3:minmaxgjl}
\PR \Big( \max_{0\leq l < [2^{j(1-\delta)}]} |G_{j,le_j}| \leq 2^{-j\frac{2\delta}{\al}} \Big) \leq \Big(1-c_5 2^{-j(1-2\delta)} \Big)^{[2^{j(1-\delta)}]},
\end{equation}
where the constant $c_5:=c_3c_4\in (0,1)$. Then,  (\ref{eq3:minmaxgjl}), the fact that $\delta\in (0,1/3)$, and standard computations, allow to show that, 
$$
\sum_{j= j_0}^{+\infty} \PR \Big( \max_{0\leq l < [2^{j(1-\delta)}]} |G_{j,le_j}| \leq 2^{-j\frac{2\delta}{\al}} \Big) <\infty;
$$
thus, applying the Borel-Cantelli Lemma, one gets (\ref{eq0:minmaxgjl}).
\end{proof}

The proof of Lemma~\ref{prop:majmaxrjl} mainly relies on the following result as well as on Lemma~\ref{lem:tailZ}.

\begin{Lem}\label{lem:majsclarjl}
For all non-negative integers $j$ and $l$ such that $l<[2^{j\delta}]$, the scale parameter $\|R_{j,le_j} \|_{\al}$ of the symmetric $\al$-stable random variable $R_{j,le_j}$ (see (\ref{def:Rjl})) satisfies,
\begin{equation}
\label{eq:scaleR}
\|R_{j,le_j} \|_{\al}^{\al} = 2^{-j} \int_{-\infty}^{1-e_j} |\Phi_{H,\al}(x)|^{\al} dx\leq c_6 2^{-j\al(2\delta +1/\al-\delta H)},
\end{equation}
where $c_6$ is a positive constant non depending on $j$ and $l$.
\end{Lem}

\begin{proof}[Proof of Lemma \ref{lem:majsclarjl}] The equality in (\ref{eq:scaleR}) can be obtained by using (\ref{def:Rjl}) and the arguments which have allowed to 
derive (\ref{eq1:lem:gjlrjl}). Let us show that the inequality in (\ref{eq:scaleR}) holds; there is no restriction to assume that $j\ge \delta^{-1}$. Using (\ref{ineg:loc}) and (\ref{eq:add1}), one has, 
\begin{eqnarray*}
&& 2^{-j}\int_{-\infty}^{1-e_j} |\Phi_{H,\al}(x)|^{\al} dx \le c_{1}^{\al} 2^{-j}\int_{-\infty}^{1-e_j}(1-x\big)^{-2\alpha-1+\alpha H} dx\\
&&\le c_{1}^{\al} 2^{-j}\int_{2^{j\delta}-2}^{+\infty} \big (1+x\big)^{-2\alpha-1+\alpha H} dx=c_{1}^{\al} \frac{2^{-j}\big (2^{j\delta}-1\big)^{-\alpha (2-H)}}{\alpha (2-H)}\le c_6 2^{-j\al(2\delta +1/\al-\delta H)},
\end{eqnarray*}
where the constant $$c_6:= c_{1}^{\al} \frac{2^{\alpha (2-H)}}{\alpha (2-H)}.$$
\end{proof}
\noindent
Now, we are in position to prove Lemma~\ref{prop:majmaxrjl}.

\begin{proof}[Proof of Lemma~\ref{prop:majmaxrjl}] 
First, observe that in view of the assumption that $\delta\in (0,1/3)$, one has for a fixed arbitrarily small $\eta>0$,
$$
\frac{2\delta+\eta}{\al}<2\delta+1/\al-\delta H;
$$
therefore, it follows from Lemma~\ref{lem:majsclarjl}, that there exists a positive integer $j_1$, such that for all integers $j$ and $l$, satisfying $j\ge j_1$ and $0\leq l < [2^{j(1-\delta)}]$,
one has,
$$
\|R_{j,le_j} \|_{\al}\le 2^{-j\big (\frac{2\delta+\eta}{\al}\big)}.
$$
Thus, we are allowed to apply the second inequality in (\ref{eq2:minmaxgjl}), in the case where $Y=R_{j,le_j}$ and $t=2^{-j\big(\frac{2\delta+\eta}{\al}\big)}$. As a consequence,
we obtain that, for all $j\ge j_1$,
\begin{eqnarray}
\label{eq1:majmaxrjl}
&& \PR\Big(\max_{0\leq l < [2^{j(1-\delta)}]} |R_{j,le_j}| > 2^{-j\big(\frac{2\delta+\eta}{\al}\big)} \Big) \leq  \sum_{l=0}^{[2^{j(1-\delta)}]-1}  \PR\Big(|R_{j,le_j}| > 2^{-j\big(\frac{2\delta+\eta}{\al}\big)} \Big)\nonumber\\
&& \le c_2 2^{j(2\delta+\eta)}\sum_{l=0}^{[2^{j(1-\delta)}]-1}\|R_{j,le_j} \|_{\al}^{\al}\le c_7 2^{-j\al(2\delta +1/\al-\delta H)+j(1+\delta+\eta)},
\end{eqnarray}
where the last inequality results from (\ref{eq:scaleR}) and the constant $c_7:=c_2 c_6$. Assume that $\delta (\al-1)>\eta$, then one has,
$$
\al(2\delta +1/\al-\delta H)>\al (\delta+1/\al)=\al\delta +1>1+\delta +\eta.
$$
Therefore, it follows from (\ref{eq1:majmaxrjl}) that,
$$
\sum_{j=j_1}^{+\infty} \PR\Big(\max_{0\leq l < [2^{j(1-\delta)}]} |R_{j,le_j}| > 2^{-j\big(\frac{2\delta+\eta}{\al}\big)} \Big) <\infty;
$$
thus, applying the Borel-Cantelli Lemma, one gets (\ref{eq0:majmaxrjl}).
\end{proof}

\begin{Rem}
\label{rem:unif-event}
Our proofs of Lemmas~\ref{prop:minmaxgjl}~and~\ref{prop:majmaxrjl}, only allow to derive that Relations (\ref{eq0:minmaxgjl})~and~(\ref{eq0:majmaxrjl}) hold on 
some event of probability $1$, denoted by $\widetilde{\Omega}_\delta$, since it a priori depends on $\delta\in (0,1/3)$. Yet, one can easily show that these two relations
also hold, for every real number $\delta\in (0,1/3)$, on an event of probability $1$ which does not depend on $\delta$, namely the event $\bigcap_{\delta\in\Q\cap
(0,1/3)} \widetilde{\Omega}_\delta$.
\end{Rem}

Now, we are in position to prove Relation (\ref{eq:mainlb}).\\
\noindent
Assume that $\epsilon$ is a fixed arbitrarily small positive real number and that $\delta\in (0,1/3)$ is such that,
\begin{equation}
\label{eq1:mainlb}
\epsilon/2=2\delta/\alpha.
\end{equation}
Next observe that (\ref{eq1:mainlb}), (\ref{eq:wavc-D}), (\ref{rel:d:GR}) and the triangle inequality, imply that for all $j\in\Z_+$,
\begin{eqnarray*}
&& 2^{j(H-1/\al +\epsilon/2)} D_j\ge 2^{j\frac{2\delta}{\alpha}} \max_{0\leq l < [2^{j(1-\delta)}]} |G_{j,le_j}+R_{j,lej}| \\
&& \geq  2^{j\frac{2\delta}{\alpha}}\max_{0\leq l < [2^{j(1-\delta)}]} |G_{j,le_j}| - 2^{j\frac{2\delta}{\alpha}}\max_{0\leq l < [2^{j(1-\delta)}]} |R_{j,le_j}|;
\end{eqnarray*}
therefore, one has that,
\begin{eqnarray}
\label{eq2:mainlb}
&& \liminf_{j\rightarrow +\infty}\left\{2^{j(H-1/\al +\epsilon/2)} D_j\right\}\\
&& \ge \liminf_{j\rightarrow +\infty}\left\{2^{j\frac{2\delta}{\alpha}}\max_{0\leq l < [2^{j(1-\delta)}]} |G_{j,le_j}|\right\}
-\limsup_{j\rightarrow +\infty}\left\{2^{j\frac{2\delta}{\alpha}}\max_{0\leq l < [2^{j(1-\delta)}]} |R_{j,le_j}|\right\}.\nonumber
\end{eqnarray}
Finally putting together, (\ref{eq0:minmaxgjl}), (\ref{eq0:majmaxrjl}), (\ref{eq2:mainlb}) and (\ref{eq1:mainlb}), one gets (\ref{eq:mainlb}).

\subsection{Proof of Theorem~\ref{Th:main}}

Relations (\ref{eq:mainub}) and (\ref{eq:mainlb}) imply that there is $\Omega^*$ an event of probability $1$ such that each $\omega\in\Omega^*$ satisfies 
the following property: for all arbitrarily small $\epsilon >0$, there are two finite positive constants 
$A=A(\omega,\epsilon)$ and $B=B(\omega,\epsilon)$, and there exists $j_2=j_2(\omega,\epsilon)\in\Z_+$, such that, one has for all integer $j\ge j_2$,
$$ 
A 2^{-j(H-1/\al+\epsilon)}\le D_j(\omega)\le B 2^{-j(H-1/\al-\epsilon)}.
$$
This entails that,
$$
-H+1/\al-\epsilon \le \liminf_{j\rightarrow +\infty}\left\{\frac{\log(D_j(\omega))}{j\log(2)}\right\}\le \limsup_{j\rightarrow +\infty}\left\{\frac{\log(D_j(\omega))}{j\log(2)}\right\}\le 
-H+1/\al+\epsilon.
$$
Then letting $\epsilon$ goes to zero, one gets that,
$$
 \lim_{j\rightarrow +\infty}\left\{\frac{\log(D_j(\omega))}{j\log(2)}\right\}=-H+1/\al.
$$
$\Box$ 
\newpage
\noindent
\\
\\

\bibliographystyle{plain}
\bibliography{mabiblio}

\begin{thebibliography}{1}

\bibitem{abry1999estimation}
P.~Abry, B.~Pesquet-Popescu, and M.~S. Taqqu.
\newblock Estimation ondelette des param\`etres de stabilit\'e et
  d'autosimilarit\'e des processus $\alpha$-stables autosimilaires.
\newblock In {\em 17\`eme Colloque sur le traitement du signal et des images,
  FRA, 1999}. GRETSI, Groupe d'Etudes du Traitement du Signal 
  1999.

\bibitem{delbeke2000stochastic}
L.~Delbeke and P.~Abry.
\newblock Stochastic integral representation and properties of the wavelet
  coefficients of linear fractional stable motion.
\newblock {\em Stochastic Processes and their Applications}, 86(2):177--182,
  2000.

\bibitem{EmMa}
P.~Embrechts and M.~Maejima.
\newblock {\em Self-Similar Processes}.
\newblock Academic Press, 2003.

\bibitem{pipiras2007bounds}
V.~Pipiras, M.~S. Taqqu, and P.~Abry.
\newblock Bounds for the covariance of functions of infinite variance stable
  random variables with applications to central limit theorems and
  wavelet-based estimation.
\newblock {\em Bernoulli}, 13(4):1091--1123, 2007.

\bibitem{SamTaq}
G.~Samorodnitsky and M.~S. Taqqu.
\newblock {\em Stable non-Gaussian random variables}.
\newblock Chapman and Hall, London, 1994.

\bibitem{SPT2002}
S.~Stoev, V.~Pipiras, and M.~S. Taqqu.
\newblock Estimation of the self-similarity parameter in linear fractional
  stable motion.
\newblock {\em Signal Processing}, 82:1873--1901, 2002.

\bibitem{stoev2005asymptotic}
S.~Stoev and M.~S. Taqqu.
\newblock Asymptotic self-similarity and wavelet estimation for long-range
  dependent fractional autoregressive integrated moving average time series
  with stable innovations.
\newblock {\em Journal of Time Series Analysis}, 26(2):211--249, 2005.

\bibitem{Tak89}
K.~Takashima.
\newblock Sample paths properties of ergodic self-similar processes.
\newblock {\em Osaka Journal of Mathematics}, 26:159--189, 1989.

\end{thebibliography}
\end{document}